\documentclass[12pt,reqno,a4paper]{amsart}
\usepackage{hyperref}
\usepackage{enumerate}
\usepackage[left=0.7in, right=0.7in, top= 0.9in, bottom=1.65in]{geometry}
\usepackage{amsmath,amssymb,amsthm,amsfonts}
\usepackage{graphicx}
\usepackage{tikz}

\usetikzlibrary{decorations.pathreplacing}
\usetikzlibrary{decorations.markings}
\usetikzlibrary{calc} % Needed for ($(v0)!0.5!(v4)$) syntax

\usetikzlibrary{graphs, graphs.standard}

\tikzset{
	modal/.style={>=stealth’,shorten >=1pt,shorten <=1pt,auto,node distance=1.5cm,
		semithick},
	world/.style={circle, draw,minimum size=.1cm,fill=gray!15},
	point/.style={circle,draw,inner sep=0.3mm,fill=black},
	circ/.style={circle,draw,inner sep=0.1mm,fill=white},
	reflexive above/.style={->,loop,looseness=7,in=120,out=60},
	reflexive below/.style={->,loop,looseness=7,in=240,out=300},
	reflexive left/.style={->,loop,looseness=7,in=150,out=210},
	reflexive right/.style={->,loop,looseness=7,in=30,out=330}
}

\usetikzlibrary{shapes}
\usetikzlibrary{plotmarks}
\usetikzlibrary{arrows}
\usetikzlibrary{positioning}
\theoremstyle{definition}
\newtheorem{defn}{Definition}[section]
\newtheorem{fact}[defn]{Fact}
\newtheorem{prop}[defn]{Proposition}
\newtheorem{thm}[defn]{Theorem}

\newtheorem{corr}[defn]{Corollary}
\newtheorem{lem}[defn]{Lemma}

\newtheorem{claim}[defn]{Claim}
\newtheorem{subclaim}[defn]{Subclaim}
\title[Cayley Graphs of Dihedral Groups of Valency 4]{Automorphism Groups and Structure of 4-Valent Cayley Graphs on Dihedral Groups}

\author{Amitayu Banerjee}
\address{E\"otv\"os Lor\'and University, Budapest, Hungary}
\email{banerjee.amitayu@gmail.com}

\date{}

\makeatletter
\@namedef{subjclassname@2020}{\textup{2020} Mathematics Subject Classification}
\makeatother
\begin{document}
\begin{abstract}

Let $G$ be a finite group and $S \subseteq G \setminus \{e\}$ an inverse-closed subset.  
The undirected Cayley graph $\mathrm{Cay}(G,S)$ has vertex set $G$, where vertices $x,y$ are adjacent when $xy^{-1}\in S$.
Kaseasbeh and Erfanian (2021) determined the structure of all $\mathrm{Cay}(D_{2n},S)$ with $|S|\le 3$, where $D_{2n}$ denotes the dihedral group of order $2n$.    
We extend this by determining the structure of all $\mathrm{Cay}(D_{2n},S)$ with $|S|=4$.
Specifically, 
\begin{enumerate}
    \item if $S$ consists of $4 \le 2k < n$ distinct rotations, then $\mathrm{Cay}(D_{2n},S)$ is the disjoint union of two isomorphic circulant graphs on $n$ vertices, and
    \item if $S$ is a generating set of $4\leq k\leq n$ reflections, then $\mathrm{Cay}(D_{2n},S)$ is bipartite and decomposes into $k$ perfect matchings. 
\end{enumerate}

Using a result of Burnside and Schur in the formulation of Evdokimov and Ponomarenko from 2005, we determine $\mathrm{Aut}(\mathrm{Cay}(D_{2p},S))$ for infinitely many primes $p$ when $S$ contains distinct rotations.
\end{abstract}
   
\maketitle

\section{Introduction}

The study of automorphism groups of Cayley graphs is one of the central topics in algebraic graph theory. Cayley graphs on dihedral groups, in particular, have received significant attention as a rich class of examples for this research  
(cf. \cite{KE2021,Kon2020,KO2006,WZ2006,WX2006,ZF2007}, among others). 
Previous research work has largely focused on
Cayley graphs with valency at most $3$. In particular, Kong \cite{Kon2020} studied the automorphism group of connected cubic Cayley graphs of dihedral groups of order $2^{n}p^{m}$ where $n\geq 2$ and $p$ is an odd prime, while
Kaseasbeh and Erfanian \cite{KE2021} determined the structure of all Cay($D_{2n}$, $S$), where $n \geq 3$ and $|S| \leq 3$.
These studies provide a foundation for understanding higher-valency cases.
%%%%%%%%%%%%%%%%%%%%%%%%%%%%%%%%%%%%%%%%
The classification of 4-valent one-regular normal Cayley graphs on dihedral groups was  investigated in \cite{KO2006,WZ2006,WX2006}. Notably, Wang and Xu \cite{WX2006} 
proved that all 4-valent one-regular Cayley graph $X$ of dihedral groups are normal except that $n=4s$, and $X\cong \mathrm{Cay}(G,\{a,a^{-1}, a^{i}b,a^{-i}b\})$ where $i^{2}\equiv \pm 1\pmod {2s}$, $2\leq i\leq 2s-2$. However, a complete understanding of all 4-valent Cayley graphs over dihedral groups, including their structural properties and automorphism groups, remains an open area.
In this paper, we extend this line of research by investigating the structure of 
$\mathrm{Cay}(D_{2n},S)$ for $|S|=4$ and the automorphism 
groups of $\mathrm{Cay}(D_{2n},S)$ for $|S|\geq 4$ when $S$ consists exclusively of rotations or reflections.

\subsection{Results}

Applying a result of Burnside and Schur from 1911 in the formulation of Evdokimov and Ponomarenko \cite{EP2005}, we prove that 
if for some
$k\ge 2$ and $t_i \ge 2$ (for each $1\le i \le k-1$)
\begin{center}
$p > \max_{a,b\in\{1,t_1,\dots,t_{k-1}\}} \bigl(ab + \max\{1, t_1, \dots, t_{k-1}\}\bigr)$    
\end{center}
is a prime, and if 
$S=\{r^{\pm1}, r^{\pm t_1}, \dots, r^{\pm t_{k-1}}\}$
contains distinct non-identity rotations of $D_{2p}$, then 
$
\mathrm{Aut}(\mathrm{Cay}(D_{2p}, S))
    \cong (R(\mathbb{Z}_p) \rtimes \langle p-1 \rangle) \wr \mathbb{Z}_2 
$ (cf. Theorem~\ref{Theorem 3.7}).
Apart from this and the results stated in the Abstract, we prove the following:

\begin{enumerate}
    \item If $S=\{r^{a_1}s,\ldots,r^{a_k}s\}$ is a generating set of $4\le k\le n$ reflections, $\Gamma=\mathrm{Cay}(D_{2n},S)$ is normal, $\gcd(k,n)=1$, and \(\Delta=\{a_i-a_j : 1\le i<j\le k\}\), then 
    $\mathrm{Aut}(\Gamma)=R(G)\rtimes H$,
    where 
    $
    H\le \{u\in(\mathbb{Z}_n)^\times : u\Delta=\Delta\}.
    $

    \item If $S$ contains two rotations and two reflections, then $\mathrm{Cay}(D_{2n},S)$ consists of two isomorphic circulants joined by two inter-layer perfect matchings.

    \item If $S$ contains three rotations and one reflection, then $\mathrm{Cay}(D_{2n},S)$ is formed by two isomorphic circulants joined by a single inter-layer perfect matching.

    \item If $S$ contains three reflections and one rotation, then $\mathrm{Cay}(D_{2n},S)$ consists of two circulants (with intra-layer edges linking vertices at distance $n/2$) joined by three inter-layer perfect matchings.
\end{enumerate}

\section{Preliminaries}

\begin{defn}\label{Definition 21}
Let $\Gamma = (V(\Gamma),E(\Gamma))$ be a graph. A \emph{matching} in $\Gamma$ is a subset $M \subseteq E(\Gamma)$ such that no two edges in $M$ share a vertex, and it is a \emph{perfect matching} if every vertex of $\Gamma$ is incident with exactly one edge in $M$.
The \emph{$n$-Crown graph} for an integer $n \ge 3$ is the bipartite graph with bipartitions 
$\{x_1,\dots,x_n\}$ and $\{y_1,\dots,y_n\}$ and edges
$\{x_i,y_j\}$ for all $i\ne j$.  
Equivalently, it is  the complete bipartite graph  $K_{n,n}$ from which the perfect matching $\{\{x_i, y_i\} : 1\le i\le n\}$ has been removed.
\end{defn}

\begin{defn}\label{Definition 2.2}
Let $G$ be a group that acts on a set $X$ such that $\vert X\vert \geq 2$. 
The action is called \emph{transitive} if for all $x,y \in X$ there exists $g \in G$ such that $gx = y$.  
It is \emph{2-transitive} if for any $x_{1},x_{2},y_{1},y_{2}\in X$ such that $x_{1}\neq x_{2}$ and $y_{1}\neq y_{2}$, there exists $g \in G$ such that 
$g x_i = y_i$ for $i = 1,2$.
Let $Orb_{G}(x)=\{gx:g\in G\}$ be the orbit of $x \in X$ and $Stab_{G}(x)=\{g\in G: gx=x\}$ be the stabilizer of $x$ under the action of $G$. 
\end{defn}

%The {\em Orbit-Stabilizer theorem} states that the size of the orbit is the index of the stabilizer, that is $\vert Orb_{G}(x)\vert = [G : Stab_{G}(x)]$. We recall that different orbits of the action are disjoint and form a partition of $X$ i.e., $X=\bigcup \{Orb_{G}(x): x\in X\}$.

\begin{defn}\label{Definition 2.3}
A group $G$ is called a \emph{semidirect product} of $N$ by $Q$, 
denoted by $G = N \rtimes Q$, if $G$ contains subgroups $N$ and $Q$ such that:
(1).  $N \trianglelefteq G$ (that is, $N$ is a normal subgroup of $G$), (2). $NQ = G$, and
(3). $N \cap Q = \{1\}$.
\end{defn}

\begin{defn}\label{Definition 2.4}
The \emph{affine group} $\mathrm{AGL}(1,n)$ is the group of functions 
$x \mapsto ax+b$ on $\mathbb{Z}_n$, where $a \in \mathbb{Z}_n^{\ast}$ and 
$b \in \mathbb{Z}_n$.  Equivalently, 
$\mathrm{AGL}(1,n) \cong \mathbb{Z}_n \rtimes \mathbb{Z}_n^{\ast}$.
\end{defn}

%\begin{defn}\label{Definition 2.4} The {\em affine group} $\mathrm{AGL}(1,n)$ is the semidirect product of the group of translations $\mathbb{Z}_n$ and the group of automorphisms $\mathrm{Aut}(\mathbb{Z}_n)$. Alternatively, it is the group of functions $x \mapsto ax+b$ where $a \in \mathbb{Z}_n^*$ and $b \in \mathbb{Z}_n$, where $\mathbb{Z}_n^*$ is the multiplicative group of integers modulo $n$ which consists of the set of integers $k$ with $1\le k<n$ such that $\gcd (k,n)=1$ and the group operation is multiplicative modulo $n$. If $n$ is a prime then $\mathbb{Z}_n^*$ contains all non-zero integer modulo $n$.\end{defn}

\begin{defn}\label{Definition 2.5}
Let $G$ be a group and $S \subseteq G \setminus \{e\}$ 
be inverse-closed i.e., $S = S^{-1}$, where 
$S^{-1}:= \{s^{-1}: s \in S\}$.
The \emph{undirected Cayley graph} Cay$(G, S)$ is the graph with a set of vertices $G$, and the vertices $u$ and $v$ are adjacent in Cay$(G, S)$ if and only if $uv^{-1} \in S$. The size of the set $S$ is called the \emph{valency} of $\mathrm{Cay}(G, S)$. It is known that Cay$(G, S)$ is connected if and only if $S$ is a generating set of $G$.
\end{defn}

\begin{defn}\label{Definition 2.6}
The \emph{right regular representation} of a group $G$, denoted by $R(G)$,
is the permutation group 
$\{\rho_g : G \to G \mid \rho_g(x)=xg 
\text{ for all }
x\in G,\; g\in G\}$ in $\mathrm{Sym}(G)$.
The automorphism group of $\mathrm{Cay}(G,S)$ is denoted by
$\mathrm{Aut}(\mathrm{Cay}(G,S))$.
\end{defn}

It is known that $R(G)$ is a subgroup of $\mathrm{Aut}(\mathrm{Cay}(G,S))$.

\begin{defn}\label{Definition 2.7}
The {\em stabilizer of vertex $v$ in $\mathrm{Aut}(\mathrm{Cay}(G,S))$} is denoted by $\mathrm{Aut}(\mathrm{Cay}(G,S))_v$. 
Given a group $G$ and a subset $S \subseteq G$, let 
$\mathrm{Aut}(G,S) = \{\alpha \in \mathrm{Aut}(G) \mid \alpha(S) = S\}$.
%It is isomorphic to the group of automorphisms of $G$ that fix the generating set $S$ setwise; i.e., $\mathrm{Aut}(\mathrm{Cay}(G,S))_e \cong \mathrm{Aut}(G,S) = \{\alpha \in \mathrm{Aut}(G) \mid \alpha(S)=S\}$.
\end{defn}

If $\Gamma = \mathrm{Cay}(G, S)$, then $\mathrm{Aut}(G,S)$ is a subgroup of the stabilizer $\mathrm{Aut}(\Gamma)_{1}$, where $1$ is the identity element of the group $G$. Moreover, $R(G)\rtimes \mathrm{Aut}(G,S)\le \mathrm{Aut}(\Gamma)$.

\begin{defn}\label{Definition 2.8}
A Cayley graph $\Gamma=$ Cay$(G,S)$ is {\em normal} if $R(G)$ is a normal subgroup of $\mathrm{Aut}(\Gamma)$ i.e., $R(G)\trianglelefteq \mathrm{Aut}(\Gamma)$. The graph $\Gamma$ is normal if and only if $\mathrm{Aut}(\Gamma)=R(G) \rtimes \mathrm{Aut}(G,S)$.
\end{defn}

\begin{fact}\label{Fact 2.9} 
The following holds:
%\cite{Xu1998, HHL2017}
    \begin{enumerate}
        \item (\cite{HHL2017}) Cay$(G,S)$ is normal if and only if $\mathrm{Aut}(\mathrm{Cay}(G,S))_{e} = \mathrm{Aut}(G,S)$.
        \item (Burnside-Schur; \cite{EP2005}) Every primitive finite permutation group containing a regular cyclic subgroup is either 2-transitive or permutationally isomorphic to a subgroup of the affine group AGL$(1,p)$ where $p$ is a prime.
        
        %\item For any integer $n>1$, $\mathrm{Aut}(\mathbb{Z}_n) \cong \mathbb{Z}_n^*$.        
        
        \item If the action of $G$ on $X$ is $2$-transitive, then the action of Stab$_{G}(x)$ on $X\backslash \{x\}$ is transitive for all $x\in X$.
    \end{enumerate}
\end{fact}

Since any transitive permutation group of prime degree is primitive, Fact \ref{Fact 2.9}(2) immediately yields the following.

\begin{corr}\label{Corollary 2.10}
{\em Let $p$ be a prime and $G\leq S_p$ be a transitive permutation group of degree $p$ that contains a regular cyclic subgroup. Then $G$ is primitive and $G$ is either
isomorphic to a subgroup of $\mathrm{AGL}(1,p)$, or $G$ is $2$-transitive.
}
\end{corr} 

Throughout the manuscript, we will use the following notations.
\begin{itemize}
    \item $D_{2n}=\langle r,s \mid r^n=e, \; s^2=e, \; srs=r^{-1} \rangle$ be the dihedral group of order $2n$,
    \item $\mathbb{Z}_{n}$ denotes cyclic group of order $n$,
    \item  
    $R_{n} = \{ r^i : i \in \mathbb{Z}_n \}$ be the set of all rotations, and 
    \item $F_{n} = \{ sr^i : i \in \mathbb{Z}_n \}$ be the set of all reflections. Thus, $D_{2n} = R_{n} \cup F_{n}$.
    \item The indices of rotations and reflections are taken modulo $n$ whenever we work with $\mathrm{Cay}(D_{2n}, S)$.
    \item We refer to edges connecting two rotations or two reflections as \emph{intra-layer edges}, and those connecting a rotation with a reflection as \emph{inter-layer edges}.

    \item For graphs $G_1$ and $G_2$,
$G_1 + G_2$ is the 
\emph{disjoint union} of $G_1$ and $G_2$.
\end{itemize}
%%%%%%%%%%%%%%%%%%%%%%%%%%%%%%%%%%%%%%%%%%%%%%%
Let $S \subset D_{2n}$ satisfy $e \notin S$, $S=S^{-1}$ and $|S|=4$. 
Then $\Gamma := \mathrm{Cay}(D_{2n}, S)$
falls into exactly one of the following mutually exclusive types:

\begin{itemize}
    \item[] \noindent\textbf{Case (I)— $S \subseteq R_{n}$.}  
Then $S = \{r^{\pm a}, r^{\pm b}\}$
for some $a,b \in \mathbb{Z}_n$ (possibly $a \equiv \pm b$ (mod $n$)).
%\vspace{2mm}
     \item[] \noindent\textbf{Case (II)— $S \subseteq F_{n}$.}  
Then $S=\{sr^{a_1}, sr^{a_2}, sr^{a_3}, sr^{a_4}\}$ for some $a_{1},a_{2},a_{3},a_{4} \in \mathbb{Z}_n$. Clearly, $S=S^{-1}$ since each reflection is an involution. 
%\vspace{2mm}
      \item[] \noindent\textbf{Case (III)—}  
$S$ contains exactly two rotations and two reflections. Then, for some $a,b_1,b_2\in\mathbb{Z}_n$,
$S=\{r^{\pm a}, sr^{b_1}, sr^{b_2}\}$. 
%\vspace{2mm}
    \item[] \noindent\textbf{Case (IV)—}  
$S$ contains exactly three rotations and one reflection. This case occurs only when $n$ is even. Then three rotations in $S$ must consist of one inverse pair and the unique element of order two, that is $r^{n/2}$. Thus,
$S = \{r^{\pm a},\, r^{n/2},\, sr^b\}$, for some $a,b\in\mathbb{Z}_n$.
%\vspace{2mm}
    \item[] \textbf{Case (V)—} 
$S$ contains exactly three reflections and one rotation. This case arises only when $n$ is even and the rotation in $S$ must be $r^{n/2}$. Hence,
$S = \{ s r^{a_1}, \; s r^{a_2}, \; s r^{a_3}, \; r^{n/2} \}$, for some $a_1,a_2,a_3 \in \mathbb{Z}_n$.
\end{itemize}

In sections 3--5, we will analyze the above-mentioned cases.

%%%%%%%%%%%%%%%%%%%%%%%%%%%%%%%%%%
\section{Only rotations}

\begin{prop}\label{Proposition 3.1}
{\em
Assume $S\subseteq R_{n}\setminus\{e\}$ is an inverse-closed subset of rotations of $D_{2n}$ with $|S|=2k<n$ for some $k\ge2$.  
Choose representatives $a_1,\dots,a_k\in\mathbb{Z}_n$ such that 
$S=\{r^{\pm a_1},\dots,r^{\pm a_k}\}$.
Let $T = \{\pm a_1, \dots, \pm a_k\}$, 
$G_{1} = \mathrm{Cay}(D_{2n}, S)$, 
$G_{2} = \mathrm{Cay}(\mathbb{Z}_n, T)$, and $d = \gcd(n, a_1, \dots, a_k)$. Then:
\begin{enumerate}
  \item $\mathrm{Cay}(D_{2n}, S) \;\cong\; \mathrm{Cay}(\mathbb{Z}_n, T) + \mathrm{Cay}(\mathbb{Z}_n, T)$.
  
  \item If $d = 1$, then $G_{2}$ is connected. So $G_{1}$ has $2$ components isomorphic to $G_{2}$.
  
  \item If $d > 1$, write $n = d n'$ and $a_i = d a_i'$ for all $i$, and set 
  $T' = \{\pm a_1', \dots, \pm a_k'\} \subset \mathbb{Z}_{n'}$. 
  Then $G_{2}$ decomposes into $d$ components isomorphic to $\mathrm{Cay}(\mathbb{Z}_{n'}, T')$. 
  So, $G_{1} \cong G_{2} + G_{2}$ splits into $2d$ components isomorphic to $\mathrm{Cay}(\mathbb{Z}_{n'}, T')$.
\end{enumerate}
}
\end{prop}

\begin{proof}
(1). The vertex set of $\mathrm{Cay}(D_{2n}, S)$ is $D_{2n} = R_{n} \cup F_{n}$.  
For any rotation $r^t \in R_{n}$, 
if $g \in R_{n}$, then $g r^t \in R_{n}$, while if $g \in F_{n}$, then $g r^t \in F_{n}$.  
Thus, every edge $\{ g, g r^t \}$ produced by a rotation generator $r^{t}\in S$ is an intra-layer edge.  
Consequently, no generator in $S$ produces an edge joining $R_{n}$ to $F_{n}$.  
Therefore, $\mathrm{Cay}(D_{2n}, S)$ splits into two vertex-disjoint subgraphs induced on $R_{n}$ and on $F_{n}$.
Consider the induced subgraph $\Gamma_{R_{n}}$ of $\mathrm{Cay}(D_{2n}, S)$ on $R_{n}$.  
If $t \in T$, then $\Gamma_{R_{n}}$ contains the edge $\{ r^i, r^{i+t} \}$.  
Hence $\Gamma_{R_{n}}\cong\mathrm{Cay}(\mathbb{Z}_n, T)$.  
Similarly, the induced subgraph $\Gamma_{F_{n}}$ on $F_{n}$ is isomorphic to $\mathrm{Cay}(\mathbb{Z}_n, T)$.  
In particular, the map $\varphi : R_{n} \to F_{n}$ defined by $\varphi(r^i) = s r^i$ is a bijection, and for any $t \in T$,
$\{ \varphi(r^i), \varphi(r^{i+t}) \}
= \{ s r^i, s r^{i+t} \}
= \{ s r^i, (s r^i) r^t \}$.
Thus, edges inside $R_{n}$ correspond exactly to edges inside $F_{n}$ under $\varphi$, so $\Gamma_{F_{n}} \cong \Gamma_{R_{n}}$.  
Consequently,
$\mathrm{Cay}(D_{2n}, S)
= \Gamma_{R_{n}} + \Gamma_{F_{n}} 
\;\cong\;
\mathrm{Cay}(\mathbb{Z}_n, T) + \mathrm{Cay}(\mathbb{Z}_n, T)$.

%%%%%%%%%%%%%%%%%%%%%%%%%%%%%%%%%%%%%%%%%%%%%
\begin{figure}[h]
\centering
\begin{tikzpicture}[
    scale=0.45,
    every node/.style={circle,draw,fill=white,inner sep=1pt,minimum size=7pt},
    line width=0.5pt
]

% Function to draw K_{2,2,2} with unique index (#1) and x-shift (#2)
\newcommand{\Ktwotwotwo}[2]{
  \def\i{#1}
  \def\xshift{#2}
  \def\dx{1.0}

  % A-part (top)
  \node (A1\i) at (\xshift-0.8*\dx, 1.6) {};
  \node (A2\i) at (\xshift+0.8*\dx, 1.6) {};

  % B-part (middle)
  \node (B1\i) at (\xshift-1.6*\dx, 0) {};
  \node (B2\i) at (\xshift+1.6*\dx, 0) {};

  % C-part (bottom)
  \node (C1\i) at (\xshift-0.8*\dx,-1.6) {};
  \node (C2\i) at (\xshift+0.8*\dx,-1.6) {};

  % Edges between parts
  \foreach \u in {A1\i,A2\i}{
    \foreach \v in {B1\i,B2\i,C1\i,C2\i}{
      \draw (\u)--(\v);
    }
  }
  \foreach \u in {B1\i,B2\i}{
    \foreach \v in {C1\i,C2\i}{
      \draw (\u)--(\v);
    }
  }
}

% ---------- Draw both ----------
\Ktwotwotwo{1}{0}    % left copy
\Ktwotwotwo{2}{6.0}  % right copy

% Labels
\node[draw=none,fill=none] at (-1,-2.7)
  {$\mathrm{Cay}(\mathbb{Z}_{6},\{\pm1,\pm2\})$};
\node[draw=none,fill=none] at (7,-2.7)
  {$\mathrm{Cay}(\mathbb{Z}_{6},\{\pm1,\pm2\})$};

\end{tikzpicture}
\vspace{-4.5em}
\caption{\em The graph $\mathrm{Cay}(D_{12}, \{r^{\pm1}, r^{\pm2}\})$ is the disjoint union $K_{2,2,2} + K_{2,2,2}$.}
\end{figure}

%%%%%%%%%%%%%%%%%%%%%%%%%%%%%%%%%%%%%%%%%%%%%

(2). We know that $\mathrm{Cay}(\mathbb{Z}_n, T)$ is connected if and only if $T$ is a generating set of $\mathbb{Z}_{n}$.  
The subgroup generated by $T$ is 
$\langle T \rangle$ = $\{ x_1 a_1 + \cdots + x_k a_k \pmod{n} : x_i \in \mathbb{Z} \}$
= $d \mathbb{Z}_n$
= $\{ 0, d, 2d, \dots, n - d \}$.  
Hence $\mathrm{Cay}(\mathbb{Z}_n, T)$ is connected if and only if $\langle T \rangle = \mathbb{Z}_n$, which is equivalent to $d=\gcd(n, a_1, \dots, a_k) = 1$.

(3). We recall that $d = \gcd(n, a_1, \dots, a_k)$, $n = d n'$, $a_i = d a_i'$ for $i = 1, \dots, k$.  
Let $\{ C_j : 0 \le j \le d-1 \}$ be a partition of $\mathbb{Z}_{n}$ where 
$C_j = \{j + k d : k = 0, \dots, n'-1\}$.
The graph $G_2$ is the disjoint union of induced subgraphs on $C_j$'s.  
In particular, if $x \in C_j$ and $t \in T$, then $t$ is a multiple of $d$ (since each $a_i$ is).  
Thus, $x + t \in C_j$ since $x + t \equiv x \pmod{d}$.  
Consequently, no edge $\{ x, x + t \}$ joins $C_i$ and $C_j$ for $i \ne j$.
Fix $0 \le j \le d-1$.  
The map 
$\psi_j : C_j \to \mathbb{Z}_{n'}, \psi_j(j + k d) \equiv k \pmod{n'}$,
is a graph isomorphism from the induced subgraph on $C_j$ to 
$\mathrm{Cay}(\mathbb{Z}_{n'}, T')$,  
where $T' = \{\pm a_1', \dots, \pm a_k'\}$.  
Therefore, there are exactly $d$ identical components, each isomorphic to $\mathrm{Cay}(\mathbb{Z}_{n'}, T')$.  
Since $\gcd(n', a_1', \dots, a_k') = 1$, $\langle T' \rangle = \mathbb{Z}_{n'}$, and thus $\mathrm{Cay}(\mathbb{Z}_{n'}, T')$ is connected.  
\end{proof}

\begin{table}[h!]
\centering
\begin{tabular}{|c|c|c|c|c|}
\hline
$n$ & $T$ & Cay$(\mathbb{Z}_{n},T)$ & S & Cay$(D_{2n},S)$\\
\hline
\hline
4 & $\{\pm 1, \pm 2\}$ & Complete graph $K_4$ & $\{r^{\pm 1}, r^{\pm 2}\}$ & $K_4 + K_{4}$\\

6 & $\{\pm 1, \pm 2\}$ & Octahedral graph ($K_{2,2,2}$) & $\{r^{\pm 1}, r^{\pm 2}\}$ & $K_{2,2,2} + K_{2,2,2}$ \\

6 & $\{\pm 1, \pm 3\}$ &  $\mathrm{Circ}(6;\{1,3\})$ & $\{r^{\pm 1}, r^{\pm 3}\}$ & $\mathrm{Circ}(6;\{1,3\}) + \mathrm{Circ}(6;\{1,3\})$ \\

8 & $\{\pm 1, \pm 3\}$ & complete bipartite graph $K_{4,4}$ & $\{r^{\pm 1}, r^{\pm 3}\}$ & $K_{4,4}+ K_{4,4}$\\
\hline
\end{tabular}
\vspace{2mm}
\caption{\em Examples of Cay$(\mathbb{Z}_{n},T)$ and Cay$(D_{2n},S)$ for $n \le 8$.}
\end{table}
%%%%%%%%%%%%%%%%%%%%%%%%%%%%%%%%%%%%%%%%

%%%%%%%%%%%%%%%%%%%%%%%%%%%%%%%%%%%%%%%%%%%%
\subsection{Automorphism groups}
\begin{lem}\label{Lemma 3.2}
{\em Let $p\geq 3$ be a prime. Let $H$ be a proper subgroup of $\mathrm{Aut}(\mathbb{Z}_p) \cong \mathbb{Z}_p^*$. Let $T$ be a generating set of $\mathbb{Z}_p$ that is invariant under the action of $H$ but not under any larger subgroup of $\mathrm{Aut}(\mathbb{Z}_p)$. 
If $\Gamma = \mathrm{Cay}(\mathbb{Z}_p,T)$ and the action of $\mathrm{Aut}(\Gamma)$ on $\mathbb{Z}_p$ is not 2-transitive, then $\mathrm{Aut}(\Gamma)\cong R(\mathbb{Z}_p) \rtimes H$.}    
\end{lem}

\begin{proof}
Denote $A = \mathrm{Aut}(\Gamma)$ and 
$R(\mathbb{Z}_p)=\{R_a:x\mapsto x+a \mid a\in\mathbb{Z}_p\}$. 

\begin{claim}\label{Claim 3.3}
    {\em $\Gamma$ is normal.}
\end{claim}

\begin{proof}
All connected Cayley graphs of $\mathbb{Z}_{p}$ are normal except the complete graph $K_{p}$ by Galois and Burnside’s theorems (cf. \cite[pg. 82]{LX2003}). The condition that $A$ is not 2-transitive effectively excludes the case $\Gamma=K_{p}$.
Thus, $\Gamma$ is normal.
We provide an alternative argument to show that $\Gamma$ is normal using Burnside-Schur's theorem.
%Recall that $R(\mathbb{Z}_{p})$ is a subgroup of $A$ and is isomorphic to $\mathbb{Z}_{p}$.
Since automorphism groups of Cayley graphs are vertex-transitive, $A$ is a transitive permutation group.
Moreover, $R(\mathbb{Z}_{p})\leq A$ is a regular cyclic subgroup of $S_{p}$ since $R(\mathbb{Z}_{p})\cong \mathbb{Z}_{p}$, each $R_{a}\in R(\mathbb{Z}_{p})$ is a permutation of $\mathbb{Z}_{p}$ and that the action of $\mathbb{Z}_{p}$ is regular (i.e., transitive and free). Since $A$ is not $2$-transitive, by Corollary \ref{Corollary 2.10}, $A$ is isomorphic to a subgroup of AGL$(1,p)
=\{x\mapsto a x + b : a\in\mathbb{Z}_p^{*},\,b\in\mathbb{Z}_p\}$.
For each $\alpha\in A$, write $\alpha(0)=b$ and set
$a=\alpha(1)-\alpha(0)\in\mathbb{Z}_p^{\times}$ (since $\alpha$ is a permutation, we have 
$\alpha(1) \neq \alpha(0)$, and so 
$a=\alpha(1)-\alpha(0) \not\equiv 0 \pmod{p}$).  
Define 
\begin{center}
$\varphi(\alpha) \in \mathrm{AGL}(1,p) \text{ by } 
\varphi(\alpha)(x) = a x + b \text{ for all } x \in \mathbb{Z}_p$.
\end{center}

Since $\varphi(\alpha)$ and $\alpha$ agree on $0$ and $1$, and since every
$x\in\mathbb{Z}_p$ can be written as $1+\cdots+1$ while $\alpha$ preserves the
cyclic order, it follows that $\varphi(\alpha)(x)=\alpha(x)$ for all
$x\in\mathbb{Z}_p$.
Thus, 
the map $\varphi(\alpha)$ coincides with $\alpha$ on all of $\mathbb{Z}_p$. 
Thus $\varphi: A \hookrightarrow \mathrm{AGL}(1,p)$ is an injective homomorphism and  
$A \le \mathrm{AGL}(1,p)$.

Write elements of $\mathrm{AGL}(1,p)$ as pairs $(u,b)$ acting by $(u,b):x\mapsto ux+b$.
The group operation is $(u_1,b_1)(u_2,b_2)=(u_1u_2,u_1b_2+b_1)$ and inverses are
$(u,b)^{-1}=(u^{-1},-u^{-1}b)$. Thus the translation $t_a:x\mapsto x+a$ is the pair $(1,a)$.
For any $(u,b)\in\mathrm{AGL}(1,p)$,
$
(u,b)(1,a)(u,b)^{-1}=(u,b)(1,a)(u^{-1},-u^{-1}b)=(1,ua),
$
which is again a translation. Hence conjugation by every element of $\mathrm{AGL}(1,p)$
preserves the set of translations, so $R(\mathbb{Z}_{p})=\{(1,a):a\in\mathbb{Z}_p\}\unlhd $ AGL$(1,p)$.
%\begin{enumerate}
    %\item Any $\alpha\in A$ can be written as a composition of a translation $R_{b}$ and an element of $A_{0}$ where for all $b\in\mathbb{Z}_{p}$, $R_{b}(x)=x+b$. Recall that $A_{0}=\{m_{c}:c\in \mathbb{Z}_{p}^{*}, cT=T\}=\mathrm{Aut}(\mathbb{Z}_{p}, T)$ where $m_{c}:x\mapsto cx$ is a map for $c\in \mathbb{Z}_{p}^{*}$. Thus, $\alpha=R_{b}(m_{c}(x))=R_{b}(cx)=cx+b$ for some $b\in \mathbb{Z}_{p}$ and $c\in \mathbb{Z}_{p}^{*}$. So, $A\subseteq$ AGL$(1,p)$. Since AGL$(1,p)\leq$ Sym$(\mathbb{Z}_{p})$ and $A\leq$ Sym$(\mathbb{Z}_{p})$, we have $A\leq$ AGL$(1,p)$.
%\end{enumerate}
Since $A\leq \mathrm{AGL}(1,p)$, and $R(\mathbb{Z}_{p})\leq A$, we have $R(\mathbb{Z}_{p})\unlhd A$. Thus, $\Gamma$ is normal.
 \end{proof}

\begin{claim}\label{Claim 3.4}
    $\mathrm{Aut}(\mathbb{Z}_{p},T)=H$.
\end{claim}

\begin{proof}
Let $A_0 = \{\alpha \in A : \alpha(0)=0\}$ denote the stabilizer of $0$ in $A$. 
Since automorphisms preserve adjacency, $\alpha(T)=T$ for all $\alpha\in A_0$. 
%Thus, $A_0 = \{\alpha \in A : \alpha(T)=T\}$.

\begin{subclaim}\label{subclaim 3.5}
{\em Let $p$ be a prime and let $T \subseteq \mathbb{Z}_p$ be a generating, inverse-closed subset.  
For each $c \in \mathbb{Z}_p^{\times}$ define $m_c : \mathbb{Z}_p \to \mathbb{Z}_p$ by $m_c(x) = cx$.  
Then $A_0 = \{\, m_c : c \in \mathbb{Z}_p^{\times},\; cT = T \,\}
   = \mathrm{Aut}(\mathbb{Z}_p, T)$.
}
\end{subclaim}

\begin{proof}
By Claim \ref{Claim 3.3} and Fact \ref{Fact 2.9}(1), $\Gamma$ is normal and $\mathrm{Aut}(\Gamma)_e=\mathrm{Aut}(\mathbb{Z}_{p},T)$. Thus, for any Cayley graph on a cyclic group of prime order, every automorphism fixing the identity element is a \emph{group automorphism}.
%Since $p$ is a prime, $\mathbb{Z}_{p}$ is a cyclic additive group of order $p$.  
Since the group automorphisms of $(\mathbb{Z}_p,+)$ are exactly the multipliers $m_c : x \mapsto cx$ with $c \in \mathbb{Z}_p^{\times}$, we have
$A_0 \subseteq \{m_c : c \in \mathbb{Z}_p^{\times}\}$.
Moreover, for any such $m_c$,
\[
\{x, x+t\} \text{ is an edge } \iff t \in T
   \iff c t \in cT
   \iff \{cx, cx + c t\} \text{ is an edge}.
\]
Thus $m_c$ is an automorphism of $\Gamma$ if and only if $cT = T$.  
Conversely, any $\alpha \in A_0$ must satisfy $\alpha(T) = T$, so $\alpha = m_c$ for some $c$ with $cT = T$.  
Therefore $A_0 = \{m_c : c \in \mathbb{Z}_p^{\times},\; cT = T\}$.
 \end{proof}

Since $\langle T\rangle=\mathbb{Z}_{p}$ where $T$ is invariant under the action of $H$ but not under any larger subgroup of $\mathrm{Aut}(\mathbb{Z}_{p})$, we have $H=\{h\in\mathbb{Z}_p^{*} : hT=T\}$.
We observe that
$\mathrm{Aut}(\mathbb{Z}_p,T)
     = \{m_h : x \mapsto h x \mid h \in H\}$.
Pick any $m_{h}$ for $h\in H$.
For all adjacent pairs $\{x,y\}$, $y-x\in T\implies m_{h}(y)-m_{h}(x)=h(y-x)\in hT=T$. Thus, $m_{h}\in \mathrm{Aut}(\mathbb{Z}_{p},T)$ as $m_{h}(0)=0$. On the other hand, if $\alpha\in \mathrm{Aut}(\mathbb{Z}_p,T)$, then $\alpha=m_{c}$ for some $c\in \mathbb{Z}_{p}^{*}$ and $cT=T$. So, $c\in H$ and $\alpha=m_{c}\in \{m_h : x \mapsto h x \mid h \in H\}$.
 \end{proof}

By claims \ref{Claim 3.3} and \ref{Claim 3.4}, we have $A=\mathrm{Aut}(\Gamma) = R(\mathbb{Z}_{p})\rtimes \mathrm{Aut}(\mathbb{Z}_{p},T)\cong R(\mathbb{Z}_{p}) \rtimes H$. This completes the proof of Lemma \ref{Lemma 3.2}.
 \end{proof}
%%%%%%%%%%%%%%%%%%%

\begin{thm}\label{Theorem 3.6}
{\em
Let $p\ge 3$ be prime. Let
$S=\{r^{\pm a_1},\dots,r^{\pm a_k}\}\subseteq R_{p}\setminus\{e\}$
and let 
$T=\{\pm a_1,\dots,\pm a_k\}$
denote the exponents of the rotations in $S$ such that:
\begin{enumerate}
  \item 
  $T$ is invariant under the action of a proper subgroup $H$ of $\mathrm{Aut}(\mathbb{Z}_p)$, but not under the action of any subgroup of $\mathrm{Aut}(\mathbb {Z}_p)$ strictly larger than $H$, and
  \item If $\Gamma = \mathrm{Cay}(\mathbb{Z}_p,T)$, then the action of $\mathrm{Aut}(\Gamma)$ on $\mathbb{Z}_p$ is not 2-transitive. 
  
  %\item $\gcd(p,a_1,\dots,a_k)=1$. 
  %(equivalently, $T$ generates $\mathbb{Z}_p$).
\end{enumerate}
Then,
$\mathrm{Aut}\bigl(\mathrm{Cay}(D_{2p},S)\bigr)\cong \bigl(R(\mathbb {Z}_p)\rtimes H\bigr)\wr \mathbb{Z}_2$
where $\wr$ denotes the wreath product.
%, i.e., the wreath product of $R(\mathbb {Z}_p)\rtimes H$ with $\mathbb{Z}_{2}$.
}
\end{thm}

\begin{proof}
Since \(p\) is prime and \(S\subseteq R_{p}\setminus\{e\}\), $a_j\not\equiv 0 \pmod p$ for any $j=1,...,k$. Hence each \(a_j\) is a nonzero element of the additive group \(\mathbb Z_p\), so \(\gcd(p,a_j)=1\) and \(\langle a_j\rangle=\mathbb Z_p\).  In particular, \(T\) generates \(\mathbb Z_p\).  So, $\Gamma$ is connected.
By Proposition \ref{Proposition 3.1}, the graph $\mathrm{Cay}(D_{2p},S)$ is the disjoint union of two components, each isomorphic to the connected circulant graph $\Gamma$.
Moreover,
$\mathrm{Aut}(\mathrm{Cay}(D_{2p},S))\cong (\mathrm{Aut}(\Gamma))\wr S_2$ where $S_{2}$ is the symmetric group on $2$ elements.
%Aut$(\mathrm{Cay}(D_{2p},S))$ is the wreath product of $\mathrm{Aut}(\Gamma)$ with the symmetric group $S_2$ on two elements, and $T$ is a generating set of $\mathbb{Z}_{p}$.
By Lemma \ref{Lemma 3.2}, we have $\mathrm{Aut}(\Gamma)\cong R(\mathbb{Z}_p)\rtimes H$.
Since $S_2\cong \mathbb{Z}_2$, we obtain 
$\mathrm{Aut}(\mathrm{Cay}(D_{2p},S))\cong (R(\mathbb{Z}_p)\rtimes H)\wr \mathbb{Z}_2$.
%$\mathrm{Aut}(\mathrm{Cay}(D_{2p},S))\cong (\mathrm{Aut}(\Gamma))\wr S_2\cong (R(\mathbb{Z}_p)\rtimes H)\wr \mathbb{Z}_2$.
 \end{proof}
%%%%%%%%%%%%%%%%%%%%%%%%%%%%%%%%%%%%%%%%%%%%%
\begin{thm}\label{Theorem 3.7}
{\em
Fix an integer $k\ge 2$, and distinct integers $t_1,\dots,t_{k-1}\ge 2$.
Let 
\begin{enumerate}
    \item $M=\max\{1,t_1,\dots,t_{k-1}\}$, 
    \item $Q=\max_{a,b\in\{1,t_1,\dots,t_{k-1}\}} \bigl( a b + M \bigr)$, and
    \item $p>Q$ be a prime.
\end{enumerate}
 
Let $T=\{\pm 1, \pm t_{1},\dots,\pm t_{k-1}\}\subset \mathbb{Z}_p$, and let $S=\{r^{\pm 1}, r^{\pm t_1},\dots,r^{\pm t_{k-1}}\}$ be a set of distinct, non-identity rotations in $D_{2p}$.
Then $\mathrm{Aut}\bigl(\mathrm{Cay}(D_{2p},S)\bigr) \cong \bigl(R(\mathbb{Z}_p)\rtimes H\bigr)\wr \mathbb{Z}_2$,
where $H=\langle p-1\rangle=\{1,p-1\}\subset\mathbb{Z}_p^\times$.

%Fix an integer $k\ge 2$.
%Let $T=\{\pm 1, \pm t_{1},\dots,\pm t_{k-%1}\}\subset\mathbb{Z}_p$ and 
%$S=\{\,r^{\pm 1}, r^{\pm t_1},\dots,r^{\pm %t_{k-1}}\,\}\subseteq R\setminus\{e\}$ for %integers $t_1,\dots,t_{k-1}\ge 2$.    
%Let $M=\max\{1,t_1,\dots,t_{k-1}\}$,
%$Q=\max_{a,b\in\{1,t_1,\dots,t_{k-1}\}} \bigl( a b + M \bigr)$ and $p>Q$ be a prime.
%Then $\mathrm{Aut}\bigl(\mathrm{Cay}%(D_{2p},S)\bigr)
%\cong \bigl(R(\mathbb{Z}_p)\rtimes H\bigr)\wr %\mathbb{Z}_2$,
%if $H=\langle p-1\rangle=\{1,p-1\}\subset\mathbb{Z}_p^\times$.
}
\end{thm}

\begin{proof}
Since $1\in T$, we have $\langle T\rangle=\langle 1\rangle=\mathbb Z_p$, so $\Gamma=\mathrm{Cay}(\mathbb Z_p,T)$ is connected.
In view of Theorem \ref{Theorem 3.6}, it is enough to show that $T$ is invariant under the action of $H$ but not under any larger subgroup of $\mathrm{Aut}(\mathbb{Z}_p)$ and 
the action of $\mathrm{Aut}(\Gamma)$ on the set of vertices is not 2-transitive. We proceed by verifying these properties.

\begin{claim}\label{Claim 3.8}
{\em $T$ is invariant under the action of $H$}.
\end{claim}

\begin{proof}
For the units $1,-1\in \mathbb{Z}_{p}^{*}$, $1T=T$ and $(-1)T=T$.
 \end{proof}

\begin{claim}\label{Claim 3.9}
{\em $T$ is not invariant under any subgroup of $\mathrm{Aut}(\mathbb{Z}_p)$ larger than $H$}.
\end{claim}

\begin{proof}
Suppose, for contradiction, there exists
$m\in\mathbb{Z}_p^\times\setminus\{1,p-1\}\text{ with } mT=T$.
As $1\in T$, we must have $m\in T$. Hence $m\in\{\pm t_a\}$ for some $a\in\{1,\dots,k-1\}$.
Consider the case $m=t_a$ (the $m=-t_a$ case is identical up to signs). Then
\[
mT=\{\,t_a,-t_a,t_a^2,-t_a^2,t_a t_b,-t_a t_b\ (b=1,\dots,k-1)\,\}.
\]
Since we have $mT=T$, each element on the left must equal (mod $p$) one of the elements of $T=\{\pm1,\pm t_1,\dots,\pm t_{k-1}\}$.
In particular, $t_a^2$ is congruent modulo $p$ to some $s\in T$.
But for every such $s$ we have 
\[
0<|t_a^2-s|\le t_a^2 + M \le Q < p.
\]
Hence, we have $t_a^2\equiv s \pmod p$.  This is impossible since $t_a^2\neq\pm1$ and  $t_a^2\neq \pm t_b$ (as $t_a^2>t_a\ge t_b$ except in degenerate coincidence which our inequality rules out).  
Similarly, each product $t_a t_b$ appearing in $mT$ cannot equal any element of $T$ by the same magnitude bound and hence cannot be congruent to an element of $T$ modulo $p$.
Therefore no such $m$ exists, a contradiction.
 \end{proof}

\begin{claim}\label{Claim 3.10}
{\em $\mathrm{Aut}(\Gamma)_{0}=\{m \in \mathbb{Z}_p^{\times} : mT = T\}=\{1,-1\}$.}
\end{claim}

\begin{proof}
By the arguments in the proof of Lemma \ref{Lemma 3.2},  $\Gamma$ 
is normal and $\mathrm{Aut}(\mathbb{Z}_p,T)=H=\{1,-1\}$.
By Fact \ref{Fact 2.9} (1),
we obtain
$\mathrm{Aut}(\Gamma)_0=\mathrm{Aut}(\mathbb{Z}_p,T)
   =\{1,-1\}$.
%where $0$ is the identity of $\mathbb{Z}_{p}$.
 \end{proof}

\begin{claim}\label{Claim 3.11}
{\em If $\Gamma = \mathrm{Cay}(\mathbb{Z}_p,T)$, the action of $\mathrm{Aut}(\Gamma)$ on $\mathbb{Z}_{p}$ is not 2-transitive.
}
\end{claim}

\begin{proof}
If the action of $A=\mathrm{Aut}(\Gamma)$ on the vertex set $\mathbb{Z}_{p}$ is 2-transitive, then for any fixed point $x_0$ the stabilizer $A_{x_0}=$Stab$_{A}(x_{0})$ acts transitively on the remaining $p-1$ vertices i.e., on all vertices of $\mathbb{Z}_{p}\backslash\{x_{0}\}$.  
Thus, for all $y_{1},y_{2}\in \mathbb{Z}_{p}\backslash\{x_{0}\}$, there exists $g\in A_{x_{0}}$ such that $g(y_{1})=y_{2}$.
Thus, Orb$_{A_{x_{0}}}(y)=\{g(y):g\in A_{x_{0}}\}=\mathbb{Z}_{p}\backslash \{x_{0}\}$; so $\vert$Orb$_{A_{x_{0}}}(y)\vert=p-1$.
By the Orbit--Stabilizer Theorem,
\[
|A_{x_0}| = |\mathrm{Orb}_{A_{x_0}}(y)| \cdot |(A_{x_0})_y| 
          = (p-1)\cdot |(A_{x_0})_y|,
\]
so $|A_{x_0}|$ is a multiple of $p-1$. 
In particular, $|A_{x_0}| \ge p-1$.
By Claim \ref{Claim 3.10}, we have $|A_{x_0}|= 2$.
Since $p>Q\geq 3$ because each $t_{i}\geq 2$, this is impossible.
 \end{proof}
 \end{proof}
%%%%%%%%%%%%%%%%%%%%%%%%%%%%%%%%%%%%%%
%\begin{corr}\label{Corollary 3.12}
%{\em $\mathrm{Aut}(\mathrm{Cay}(D_{2p}, S))\cong (R(\mathbb{Z}_{p})\rtimes H) \wr \mathbb{Z}_{2}$ if $p > 5$ is a prime, $H = \langle p-1 \rangle = \{1,\, p-1\} \subset \mathbb{Z}_{p}^{\times}$, and $S=\{r,r^{p-1},r^{2}, r^{p-2}\}$.}
%\end{corr}

%%%%%%%%%%%%%%%%%%%%%%%%%%%%%%%%%%%%%%%
\section{Only reflections}

Ahmad Fadzil--Sarmin--Erfanian \cite[Proposition 2]{ASE2019} proved that if $n\geq 3$ and $S$ contains all $n$ reflections of $D_{2n}$, then $\text{Cay}(D_{2n}, S)=K_{n,n}$. 

\begin{prop}\label{Proposition 4.1}
{\em
Fix $n,k\ge 4$, and let $S\subseteq D_{2n}$ be a set of $k$ reflections.  
Then $\mathrm{Cay}(D_{2n},S)$ is complete bipartite if and only if $k=n$ and $S$ consists of all reflections of $D_{2n}$.  
In this case, $\mathrm{Cay}(D_{2n},S)\cong K_{n,n}$.}
\end{prop}

\begin{proof}
Suppose $\Gamma=(V(\Gamma),E(\Gamma))=\text{Cay}(D_{2n}, S)$ is complete bipartite, say $K_{m_{1},m_{2}}$. Since $\Gamma$ is $k$-regular, $m_{1}=m_{2}=k$. Thus, $\vert V(\Gamma)\vert=2n=2k$. 
Since $\Gamma=K_{k,k}$, every vertex in $R_{n}$ is adjacent to every vertex of $F_{n}$. The neighbors of the identity $e$ are the generators in $S$. For $e$ to be connected to all $k$ reflections $s,sr,...,sr^{k-1}$, the set $S$ must contain all of these reflections.
Conversely, if $S$ contains all reflections, then \cite[Proposition 2]{ASE2019}
implies $\text{\em Cay}(D_{2n},S)\cong K_{n,n}$.
 \end{proof}
%%%%%%%%%%%%%%%%%%%%%%%%%%%%%%%%%%%
We generalize \cite[Proposition 2]{ASE2019} due to Ahmad Fadzil, Sarmin, and Erfanian.
%%%%%%%%%%%%%%%%%%%%%%%%%%%%%%%%%
\begin{prop}\label{Proposition 4.2}
{\em Fix $k\leq n$. Let \(S=\{sr^{a_{1}},\dots ,sr^{a_{k}}\}\subseteq D_{2n}\) be a generating set of distinct reflections. Let \(M_{a_{j}}=\bigl\{\{r^{i},sr^{\,a_{j}-i}\}:i\in \mathbb{Z}_{n}\bigr\}\) be a collection of edges for each \(j\in \{1,\dots ,k\}\) and \(\Gamma =\mathrm{Cay}(D_{2n},S)\). The following holds:
\begin{enumerate}
\item Each \(M_{a_{j}}\) is a perfect matching between $R_{n}$ and $F_{n}$.
\item The matchings $M_{a_j}$ and $M_{a_\ell}$ are edge-disjoint whenever $a_j\not\equiv a_\ell\pmod n$.

\item $\Gamma$ is bipartite with bipartitions $R_{n}$ and $F_{n}$, and its edge set decomposes as 
$
E(\Gamma) = \bigcup_{j=1}^{k} M_{a_j},
$
the disjoint union of $k$ perfect matchings.
\end{enumerate}    
}
\end{prop}

\begin{proof}
(1).
Fix $1\leq j\leq k$. Consider the bijection
$\varphi_j : R_{n} \to F_{n}$ given by $\varphi_j(r^i) = r^{i}sr^{a_{j}} = sr^{a_j-i}$.    
Then $\varphi_j^{-1}(sr^k) = r^{a_j-k}$.
Thus $M_{a_j}$ pairs each $r^i \in R_{n}$ with $\varphi_j(r^i) \in F_{n}$, and every vertex of $R_{n}$ and $F_{n}$ appears in exactly one pair. 
%So, $M_{a_j}$ is a perfect matching between $R$ and $F$.

(2). For the sake of contradiction, suppose $\{r^i,sr^{a_j-i}\}=\{r^t,sr^{a_l-t}\}$ for some $i,t\in \mathbb{Z}_{n}$.
Then $r^i=r^t$ and thus $i\equiv t\pmod n$. 
Furthermore, 
$sr^{a_j-i}=sr^{a_l-t}$ implies $a_{j}-i\equiv a_{l}-t \pmod n$.
Thus, $a_j\equiv a_\ell\pmod n$. 
%Similarly the other identification leads to the same conclusion. 

(3). 
In order to show that $E(\Gamma) = \bigcup_{j=1}^k M_{a_j}$, it suffices to show that $E(\Gamma) \subseteq \bigcup_{j=1}^k M_{a_j}$ and $M_{a_j}\subseteq E(\Gamma)$ for each $1\leq j\leq k$.

\begin{claim}\label{Claim 4.3}
{\em $E(\Gamma) \subseteq \bigcup_{j=1}^k M_{a_j}$.}
\end{claim}

\begin{proof}
By the definition of $\Gamma$, for any $g \in D_{2n}$ and $x \in S$ there is an edge $\{ g, gx \}$.  
If $x = sr^{a_j}$ and $g = r^i \in R_{n}$, then
$gx = r^i (sr^{a_j}) = sr^{a_j-i} \in F_{n}$,
so the edge $\{ r^i, sr^{a_j-i}\}$ lies in $M_{a_j}$.  
If the edge starts from a reflection vertex $g=sr^{k}\in F_{n}$, and is generated by $x=sr^{a_{j}}\in S$, then $gx=(sr^{k})(sr^{a_{j}})=(sr^{k}s)r^{a_{j}}=(sr^{k}s^{-1})r^{a_{j}}=r^{-k}r^{a_{j}} =r^{a_{j}-k}\in R_{n}$.
We claim that $\{sr^{k},r^{a_{j}-k}\}\in M_{a_{j}}$.
Let \(i=a_{j}-k\) and consider $r^{i}\in R_{n}$. Since $sr^{a_{j}-i}=sr^{a_{j}-(a_{j}-k)}=sr^{k}$, the edge $\{r^{a_{j}-k},sr^{k}\}$ lies in $M_{a_{j}}$, and this is the same edge as $\{sr^{k},r^{a_{j}-k}\}$. 
 \end{proof}

\begin{claim}\label{Claim 4.4}
{\em $M_{a_j}\subseteq E(\Gamma)$ for each $1\leq j\leq k$}.    
\end{claim}

\begin{proof}
In $\Gamma$, two vertices $g,h$ are adjacent
if $g^{-1}h\in S$.  
For each $i$ and each $j$, we have
$
(r^i)^{-1}(s r^{a_j-i})=r^{-i}s r^{a_j-i}
= r^{-i}\bigl(r^{-(a_j-i)}s\bigr)=r^{-a_j}s=s r^{a_j}\in S,
$
using $s r^t=r^{-t}s$. 
Hence $\{r^{i},\, sr^{a_j-i}\}\in E(\Gamma)$, and thus $M_{a_j}\subseteq E(\Gamma)$.
%By the definition of a Cayley graph, an edge exists between \(r^{i}\) and  $r^{i}(sr^{a_{j}})=sr^{a_{j}-i}$ because \(sr^{a_{j}}\in S\). So, any \(\{r^{i},sr^{a_{j}-i}\}\in M_{a_j}$ is generated by multiplying the element \(r^{i}\in D_{2n}\) by the generator \(sr^{a_{j}}\in S\). So, \(\{r^{i},sr^{a_{j}-i}\}\in E(\Gamma)$.
%Therefore, it is an edge of $\Gamma$.
 \end{proof}

So all edges of $\Gamma$ are between $R_{n}$ and $F_{n}$ and $\Gamma$ is bipartite.
\medskip
 \end{proof}

%%%%%%%%%%%%%%%%%%%%%%%%%%%%%%%%%%%%%%%%
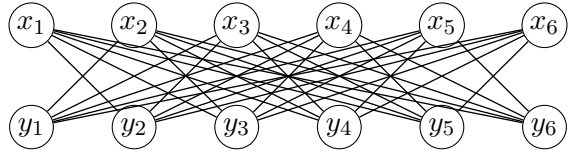
\begin{figure}[h]
\centering
\begin{tikzpicture}[
    scale=0.9,
    every node/.style={circle,draw,fill=white,inner sep=1.5pt,minimum size=10pt,font=\small},
    edge/.style={line width=0.5pt}
]

% Number of vertices per part
\def\n{6}

% Horizontal spacing
\def\dx{1.5}
% Vertical separation between x and y parts
\def\dy{1.5}

% --- Top row: x_1,...,x_6 ---
\foreach \i in {1,...,\n} {
  \node (x\i) at (\i*\dx, \dy) {$x_{\i}$};
}

% --- Bottom row: y_1,...,y_6 ---
\foreach \i in {1,...,\n} {
  \node (y\i) at (\i*\dx, 0) {$y_{\i}$};
}

% --- Edges: all x_i--y_j except x_i--y_i ---
\foreach \i in {1,...,\n} {
  \foreach \j in {1,...,\n} {
    \ifnum\i=\j
      % skip perfect matching edge
    \else
      \draw[edge] (x\i) -- (y\j);
    \fi
  }
}

\end{tikzpicture}

\caption{\em The graph $\mathrm{Cay}(D_{12},S)$ is the 6-Crown graph if $S$ consists of any $5$ reflections.}
\end{figure}
%%%%%%%%%%%%%%%%%%%%%%%%%%%%%%%%%%%%%%%%%%%%%%%
\begin{corr}\label{Corollary 4.5}
{\em Let
$S$ be a set of $n-1$ reflections from $D_{2n}$. Then $\mathrm{Cay}(D_{2n}, S)\cong$ n-Crown graph.
}
\end{corr}

\begin{proof}
Let
$S = F_{n}\setminus\{sr^{\,i_0}\}$ for some $i_{0}\in \mathbb{Z}_{n}$ where $F_{n}=\{sr^k: k\in\mathbb Z_n\}$.
Since $sr^{i_0+2},sr^{i_0+1}\in S$, their product $(sr^{i_0+1})(sr^{i_0+2})=r$ lies in $\langle S\rangle$, so $\langle S\rangle$ contains the full rotation subgroup $\langle r\rangle$. Hence $S$ generates $D_{2n}$.
By Proposition \ref{Proposition 4.2}, $\Gamma=\mathrm{Cay}(D_{2n}, S)$ is the union of $n-1$ perfect matchings between $R_{n}$ and $F_{n}$, that is, the $n$-Crown graph.
 \end{proof}
%%%%%%%%%%%%%%%%%%%%%%%%%%%%%%%%%%%%%%

\subsection{Automorphism groups}

The next lemma relates the automorphism group of a normal, connected Cayley graph
$\mathrm{Cay}(D_{2n},S)$ with $S$ consisting of only reflections,
to the stabilizer of the exponents of the elements of $S$ under the action of $\mathrm{AGL}(1,n)$.

\begin{lem}\label{Lemma 4.6}
{\em Let $n \ge 5$ be any integer and $4\leq k<n$. 
Let $S = \{r^{a_{1}}s,... r^{a_{k}}s\} \subseteq D_{2n}$
be a set of distinct reflections, 
$A = \{a_{1}, ..., a_{k}\} \subseteq \mathbb{Z}_{n}$, and 
$\Delta = \{a_{i} - a_{j} \mid 1 \le i,j \le k\}$.
Assume that the following hold:
\begin{enumerate}[(i)]
    \item \(\Gamma = \mathrm{Cay}(D_{2n}, S)\) is normal,
    \item $d=\gcd(n,a_i-a_j, 1\leq i<j\leq k)=1$.
\end{enumerate}
Then
$
\mathrm{Aut}(\Gamma)
\cong R(D_{2n}) \rtimes K$ where $K=
\bigl\{(u,v) \in (\mathbb{Z}_{n})^{\times} \ltimes \mathbb{Z}_{n}: uA + v = A\bigr\}
$.
}
\end{lem}

\begin{proof}
Since $d=1$, $S$ generates $D_{2n}$ and $\Gamma$ is connected. Since $S$ contains only reflections, $S$ is symmetric.

\begin{claim}\label{Claim 4.7}
{\em $\mathrm{Aut}(D_{2n},S)\cong K$,
i.e. the stabiliser of $A$ in the affine group $\mathrm{AGL}(1,n)$.}
\end{claim}

\begin{proof}
Given
$\psi_{u,v}: r\mapsto r^u, s\mapsto r^v s$, there is a natural correspondence
\[
\Phi:\mathrm{Aut}(D_{2n},S)\longrightarrow \mathrm{AGL}(1,n),
\qquad 
\psi_{u,v}\longmapsto (u,v).
\]
We show that
$\psi_{u,v}\in \mathrm{Aut}(D_{2n},S)
\Longleftrightarrow uA+v=A$.
%,that is, the affine map $x\mapsto ux+v$ stabilises $A$.
For $a_{i}\in A$, we have $\psi_{u,v}(r^{a_i}s)
 = (r^u)^{a_i} \, r^v s
 = r^{u a_i + v}s$.
Hence $\psi_{u,v}(S)
 = \{r^{u a_i + v}s : a_i \in A\}
 = \{r^{x}s : x \in uA + v\}$,
where $uA + v = \{u a + v : a \in A\} \subseteq \mathbb{Z}_n$.
So,
$
\psi_{u,v}(S) = S
\Longleftrightarrow 
\{r^{x}s : x \in uA + v\}
   = \{r^{x}s : x \in A\}
\Longleftrightarrow
uA + v = A$.
%Thus, $\psi_{u,v} \in \mathrm{Aut}(D_{2n},S)\Longleftrightarrow uA + v = A$.
Hence, $\mathrm{Aut}(D_{2n},S)=
\{\psi_{u,v}:(u,v)\in(\mathbb Z_n)^\times\ltimes\mathbb Z_n,\; uA+v=A\}$
%Hence $\Phi$ identifies $\mathrm{Aut}(D_{2n},S)$ isomorphically with the  stabiliser of $A$ in $\mathrm{AGL}(1,n)$, i.e., 
and $\mathrm{Aut}(D_{2n},S)\cong K$.
 \end{proof}

Since $\Gamma$ is normal, $\mathrm{Aut}(\Gamma)\cong R(D_{2n})\rtimes K$ by Claim \ref{Claim 4.7}.
 \end{proof}
%%%%%%%%%%%%%%%%%%%%%%%%%%%%%%%%%%%%%%%%
\begin{thm}\label{Theorem 4.8}
{\em
Let $4\leq k<n$ be integers such that $\gcd(n,k)=1$. Let
$
S=\{r^{a_1}s,\dots,r^{a_k}s\}\subseteq G=D_{2n}
$
be a set of distinct reflections,
and $\Delta=\{a_i-a_j:1\le i<j\le k\}$.
Assume
\begin{enumerate}[(i)]
  \item $\Gamma=\mathrm{Cay}(D_{2n},S)$ is normal,
  \item $d=\gcd(n,a_i-a_j:1\leq i<j\leq k)=1$.
\end{enumerate}
Then
$\mathrm{Aut}(\Gamma)=R(D_{2n})\rtimes H$,
such that $H\leq (U_{0},\times)$ where $U_0:=\{u\in(\mathbb Z_n)^\times : u\Delta=\Delta\}$ and $\times$ is multiplication modulo $n$.
}
\end{thm}

\begin{proof}
Since $\Gamma$ is normal, we have $\mathrm{Aut}(\Gamma)=R(D_{2n})\rtimes\mathrm{Aut}(D_{2n},S)$. Let $A=\{a_1,\dots,a_k\}$. By the arguments in the proof of Lemma \ref{Lemma 4.6}, if \(\psi_{u,v}\) maps $r\mapsto r^u$ and $s\mapsto r^v s$ then
%\begin{center}
$\mathrm{Aut}(D_{2n},S)=
\{\psi_{u,v}:(u,v)\in(\mathbb Z_n)^\times\ltimes\mathbb Z_n,\; uA+v=A\}$.    
%\end{center}
Let \(\pi:(\mathrm{Aut}(D_{2n},S),\circ)\to(\mathbb Z_n)^\times\) be the function that maps \(\psi_{u,v}\mapsto u\) where $\circ$ is the operation defined by $\psi_{u_{1},v_{1}}\circ\psi_{u_{2},v_{2}}=\psi_{u_{1}u_{2},v_{1}+u_{1}v_{2}}$ for any  $\psi_{u_{1},v_{1}},\psi_{u_{2},v_{2}}\in\mathrm{Aut}(D_{2n},S)$. 
Clearly, $\pi$ is a homomorphism.
%Since $\pi(\psi_{u_{1},v_{1}}\circ\psi_{u_{2},v_{2}}=
%u_{1}u_{2}=\pi(\psi_{u_{1},v_{1}})\pi(\psi_{u_{2},v_{2}})$, $\pi$ is a homomorphism.

\begin{claim}\label{Claim 4.9}
{\em $\pi(\mathrm{Aut}(D_{2n},S))\subseteq U_{0}$.
}
\end{claim}

\begin{proof}
If \(\psi_{u,v}\in\mathrm{Aut}(D_{2n},S)\) then $uA + v = A$. 
Thus, for every $x \in A$, there exists $y \in A$ such that $ux + v = y$, 
and conversely, for every $y \in A$, there exists $x \in A$ satisfying $ux + v = y$.
For any $x,y\in A$, there exists $x',y'\in A$ such that $x'=ux+v$ and $y'=uy+v$. Thus,
$x'- y'= (ux+v)-(uy+v)=u(x-y)$. Furthermore, $x'-y'\in\Delta$.
Thus, $u \delta \in \Delta$ for all $\delta \in \Delta$. 
So, $u\Delta \subseteq \Delta
\text{ where } 
u\Delta := \{\,u \delta : \delta \in \Delta\,\}$.
Since $u$ is a unit in $\mathbb{Z}_n$, it has a multiplicative inverse $u^{-1}\in (\mathbb{Z}_n)^\times$.  
By the same reasoning as above, we can see that
$u^{-1}\Delta \subseteq \Delta$. Multiplying both sides by $u$, we obtain $\Delta \subseteq u\Delta$.
Consequently, $u\Delta = \Delta$.
 \end{proof}

\begin{claim}\label{Claim 4.10}
{\em The kernel of $\pi$ is trivial, and so $\pi$ is an injective homomorphism.}
\end{claim}

\begin{proof}
The kernel of $\pi$ is $\mathrm{ker}(\pi)=\{\psi_{1,v}: \psi_{1,v}\in \mathrm{Aut}(D_{2n},S)\}$. If $\psi_{1,v}\in \mathrm{ker}(\pi)$, then $A+v=A$.  
Let $\mathcal{G}=(\mathbb{Z}_{n},+)$. Fix $v\in\mathbb Z_n$.  
Let $\langle v\rangle\le\mathcal G$ denote the cyclic subgroup generated by $v$, and let $\langle v\rangle$ act on $\mathbb Z_n$ by translations
$k\cdot x \equiv x + kv \pmod n, (k\in\mathbb Z)$.
%%%%%%%%%%%%%%%%%%%%%%%%%%%%%%%%%%%%%
\begin{subclaim}\label{subclaim 4.11}
{\em 
Let $m$ be the order of $v$ in $\mathcal G$, i.e. the smallest positive integer with $m v\equiv 0\pmod n$.
Then for every $x\in\mathbb Z_n$,
$|\mathrm{Orb}_{\langle v\rangle}(x)| = m$ and $m|n$.}
\end{subclaim}

\begin{proof}
The subgroup $\langle v\rangle=\{0,v,\dots,(m-1)v\}$ has $m$ elements and
$\mathrm{Orb}_{\langle v\rangle}(x)=\{x + g : g\in \langle v\rangle\}$.
Since $m$ is the least positive integer with $m v\equiv 0\pmod n$, the elements $x, x+v,\dots,x+(m-1)v$ are all distinct and
$x + m v \equiv x \pmod n$.
Thus, $|\mathrm{Orb}_{\langle v\rangle}(x)|=m$.
Since $\langle v\rangle\leq (\mathbb Z_n,+)$, Lagrange's theorem yields $m\mid n$.
 \end{proof}

Since $A=\bigcup_{x\in A} \{\mathrm{Orb}_{\langle v\rangle}(x)\}$ is a disjoint union of orbits, we have $|A|=tm=k$ if $A$ is the union of $t$-orbits, hence $m|k$. By Subclaim \ref{subclaim 4.11}, $m|n$ and thus $m=1$ since we assumed $\gcd (n,k)=1$. Therefore, $v=0$. So, the identity automorphism $\psi_{1,0}$ is the only element in $\ker(\pi)$.
 \end{proof}

\begin{claim}\label{Claim 4.12}
    {\em $(\mathrm{Aut}(D_{2n},S),\circ) \cong (\pi(\mathrm{Aut}(D_{2n},S)),\times)\leq (U_{0},\times)$.}
\end{claim}

\begin{proof}
By the arguments of Claim \ref{Claim 4.10}, there exists a unique $v(u)$ such that $uA + v(u) = A$ for any $u \in \pi(\mathrm{Aut}(D_{2n},S))$.\footnote{Indeed, if $uA+v=uA+v'$ then $uA = uA + (v-v')$. Applying $u^{-1}$ elementwise to both sides yields
$A = A + u^{-1}(v-v')$. Thus $u^{-1}(v-v')$ fixes $A$, and the orbit-count argument (as in Claim \ref{Claim 4.10}) forces $u^{-1}(v-v')\equiv 0\pmod n$, hence $v=v'$.}
Then $\psi_{u,v(u)} \in \mathrm{Aut}(D_{2n},S)$ and $\pi(\psi_{u,v(u)}) = u$, so $\pi: \mathrm{Aut}(D_{2n},S)\rightarrow \pi(\mathrm{Aut}(D_{2n},S))$ is surjective. By Claim \ref{Claim 4.10}, $\pi$ is an isomorphism.
Thus, 
$(\mathrm{Aut}(D_{2n},S),\circ) \cong (\pi(\mathrm{Aut}(D_{2n},S)),\times)$.
By Claim \ref{Claim 4.9}, we have $(\pi(\mathrm{Aut}(D_{2n},S)),\times)\leq (U_{0},\times)$.
 \end{proof}
This completes the proof of Theorem \ref{Theorem 4.8}.
 \end{proof}

%%%%%%%%%%%%%%%%%%%%%%%%%%%%%%%%
\section{Sets with Both Rotations and Reflections}

\begin{prop}\label{Proposition 5.1}
{\em
Suppose there exist integers $a,b_1,b_2 \in \mathbb{Z}_n$ such that
$S = \{ r^{\pm a},\, sr^{b_1},\, sr^{b_2} \}$ where $a\not\equiv 0, n/2 \pmod n$. 
Let $T=\{\pm a\} \subset \mathbb{Z}_n$, 
$M_{b_j} = \bigl\{ \{r^i,\, sr^{b_j-i}\} : i \in \mathbb{Z}_n \bigr\}$ for $j\in\{1,2\}$, and $\Gamma = \mathrm{Cay}(D_{2n},S)$.
Then 
\begin{center} 
    $V(\Gamma)=R_{n}\cup F_{n},\quad and \quad E(\Gamma) = E(\mathrm{Cay}(\mathbb{Z}_n,T)) \cup E(\mathrm{Cay}(\mathbb{Z}_n,T)) \cup M_{b_1} \cup M_{b_2}$.
\end{center}

In particular, $\Gamma$ is obtained by taking two identical circulant layers (on $R_{n}$ and $F_{n}$) and adding the two inter-layer perfect matchings $M_{b_1},M_{b_2}$.
}
\end{prop}

\begin{proof}
Firstly, $R_{n}$ and $F_{n}$ partition $D_{2n}$ into the rotation and reflection cosets. 
For $i \in \mathbb{Z}_n$,
$\{ r^i, r^{i+a} \}$ and $\{ r^i, r^{i-a}\}$
are edges of $\Gamma[R_{n}]$, so $\Gamma[R_{n}] \cong \mathrm{Cay}(\mathbb{Z}_n, \{\pm a\})$. Similarly, $\Gamma[F_{n}]\cong \mathrm{Cay}(\mathbb{Z}_n, \{\pm a\})$.
For $j=1,2$, each reflection $sr^{b_j}$ pairs $r^i$ with 
$r^{i}(sr^{b_j})=(r^{i}s)r^{b_{j}}=(sr^{-i})r^{b_{j}}=sr^{b_{j}-i}$,
producing the edge set
$M_{b_j} = \{ \{ r^i, sr^{\,b_j-i}\} : i \in \mathbb{Z}_n\}$.
Thus,
$E(\Gamma) = E(\Gamma[R_{n}]) \cup E(\Gamma[F_{n}]) \cup M_{b_1} \cup M_{b_2}$.
 \end{proof}
%%%%%%%%%%%%%%%%%%%%%%%%%%%%%%%%%%%%%%%

%%%%%%%%%%%%%%%%%%%%%%%%%%%%%%%%%%%%%%
\begin{figure}[h]
\centering
\begin{tikzpicture}[
    scale=0.55,
    vertex/.style={circle, draw=black, fill=gray!15, minimum size=6mm, inner sep=1pt},
    rot_edge/.style={thick, black, densely dotted},
    ref1_edge/.style={thin, black, opacity=0.7},
    ref2_edge/.style={thick, black, opacity=0.7}
]

    % R layer vertices (top row)
    \foreach \i in {0,...,5} {
        \node[vertex] (R\i) at (\i*1.8, 0) {$r^{\i}$};
    }

    % F layer vertices (bottom row)
    \foreach \i in {0,...,5} {
        \node[vertex] (F\i) at (\i*1.8, -2.5) {$sr^{\i}$};
    }

    % --- Intra-layer edges for generator r (two 6-cycles) ---
    % Rotation layer
    \draw[rot_edge] (R0) -- (R1) -- (R2) -- (R3) -- (R4) -- (R5);
    \draw[rot_edge, bend right=20] (R5) to (R0); % Curved edge to close the cycle

    % Reflection layer
    \draw[rot_edge] (F0) -- (F1) -- (F2) -- (F3) -- (F4) -- (F5);
    \draw[rot_edge, bend left=20] (F5) to (F0); % Curved edge to close the cycle

    % --- Inter-layer edges (Matchings M_b1 and M_b2) ---
    % ... (matchings code remains the same as before) ...
    % M_b1 (b1=2): Connects r^i to sr^{2-i mod 6}
    \draw[ref1_edge] (R0) -- (F2);
    \draw[ref1_edge] (R1) -- (F1);
    \draw[ref1_edge] (R2) -- (F0);
    \draw[ref1_edge] (R3) -- (F5);
    \draw[ref1_edge] (R4) -- (F4);
    \draw[ref1_edge] (R5) -- (F3);

    % M_b2 (b2=3): Connects r^i to sr^{3-i mod 6}
    \draw[ref2_edge] (R0) -- (F3);
    \draw[ref2_edge] (R1) -- (F2);
    \draw[ref2_edge] (R2) -- (F1);
    \draw[ref2_edge] (R3) -- (F0);
    \draw[ref2_edge] (R4) -- (F5);
    \draw[ref2_edge] (R5) -- (F4);

     %--- Add captions and legend ---
    %\node[left=0.5cm of R2, font=\bfseries, align=right] {Rotation Layer ($R$)};
    %\node[left=0.5cm of F2, font=\bfseries, align=right] {Reflection Layer ($F$)};

    %\begin{scope}[shift={(12,0)}]
    %    \node[anchor=west] at (-2, 0) {Edges based on following generators:};
    %    \draw[rot_edge] (-2, -1) -- (-1.2, -1) node[right] {Generator $r, r^{-1}$ (Intra-layer)};
    %    \draw[ref1_edge] (-2, -2) -- (-1.2, -2) node[right] {Generator $sr^2$ (Inter-layer)};
    %    \draw[ref2_edge] (-2, -3) -- (-1.2, -3) node[right] {Generator $sr^3$ (Inter-layer)};
    %\end{scope}
\end{tikzpicture}

\caption{\em The graph $\mathrm{Cay}(D_{12},S)$ where $S=\{r,r^{-1},sr^{2},sr^{3}\}$.}
\end{figure}
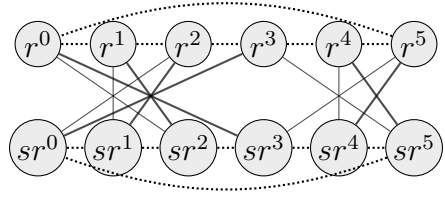

%%%%%%%%%%%%%%%%%%%%%%%%%%%%%%%%%%%
\begin{prop}\label{Proposition 6.1}
{\em Suppose $n$ is even, and there exist integers $a,b \in \mathbb{Z}_n$ such that
\begin{center}
$S=\{r^a,\, r^{-a},\, r^{n/2},\, sr^b\}$     
\end{center}

where $a\not\equiv 0, n/2 \pmod n$.
Let $\Gamma=Cay(D_{2n},S)$,
$T=\{\pm a,\, n/2\} \subset \mathbb{Z}_n$, and
$M_b=\{\{r^i,\, sr^{b-i}\} : i \in \mathbb{Z}_n\}$.
Then $E(\Gamma)= E(\mathrm{Cay}(\mathbb{Z}_n, T)) \cup E(\mathrm{Cay}(\mathbb{Z}_n, T)) \cup M_b$. 
In particular, $\Gamma$ is formed by two
isomorphic circulant graphs connected by a single inter-layer perfect matching.
}
\end{prop}

\begin{proof}
For any $r^t\in R_{n}$ and any $g\in D_{2n}$, $gr^t$ and $g$ belong to the same coset, and for any $sr^u\in F_{n}$ and any $g\in D_{2n}$, $g(sr^u)$ and $g$ belongs to the opposite coset. Thus, the generators $r^a, r^{-a}$, and $r^{n/2}$ produce only intra-layer edges. In particular, for every $i\in \mathbb{Z}_n$,
\begin{center}
$\{r^i, r^{i\pm a}\}, \{r^i, r^{i+n/2}\}, \{sr^i, sr^{i\pm a}\}, \{sr^i, sr^{i+n/2}\} \in E(\Gamma)$.
\end{center}
Therefore, $\Gamma[R_{n}] \cong \Gamma[F_{n}] \cong \mathrm{Cay}(\mathbb{Z}_n,T)$.
Furthermore, $sr^b \in S$ produces the inter-layer edges. For each $i\in \mathbb{Z}_n$,
$r^i (sr^b) = sr^{b-i} \in F$,
so the edges arising from $sr^b$ are 
$\{r^i, sr^{b-i}\}$ for $i\in \mathbb{Z}_n$.
The set $M_b=\{\{r^i, sr^{b-i}\}: i\in \mathbb{Z}_n\}$ is a perfect matching.
Consequently,
$E(\Gamma )=E(\Gamma [R_{n}])\cup E(\Gamma [F_{n}])\cup M_{b}$. As $\Gamma [R_{n}]\cong \mathrm{Cay}(\mathbb{Z}_{n},T)$ and $\Gamma [F_{n}]\cong \mathrm{Cay}(\mathbb{Z}_{n},T)$, the graph structure can be described as two identical circulant graphs connected by a perfect matching.
 \end{proof}
%%%%%%%%%%%%%%%%%%%%%%%%%%%%%%%%%%%%%%%%%%%%%%%
\begin{prop}\label{Proposition 6.2}
{\em 
Suppose \(n\) is even and there exist distinct integers \(a_1,a_2,a_3\in\mathbb{Z}_n\) such that 
$S=\{sr^{a_1},\; sr^{a_2},\; sr^{a_3},\; r^{n/2}\}$.
Let \(\Gamma=\mathrm{Cay}(D_{2n},S)\).
For \(j=1,2,3\), define $M_{a_j}=\bigl\{\{r^i,\,sr^{a_j-i}\}:i\in\mathbb{Z}_n\bigr\}$,
$N_{R_{n}}=\bigl\{\{r^i,r^{\,i+n/2}\}:i\in\mathbb{Z}_n\bigr\}$, and $N_{F_{n}}=\bigl\{\{sr^i,sr^{\,i+n/2}\}:i\in\mathbb{Z}_n\bigr\}$.
Then $E(\Gamma) = N_{R_{n}}\cup N_{F_{n}}\cup M_{a_1}\cup M_{a_2}\cup M_{a_3}$.}
\end{prop}

\begin{proof} 
For \(j\in \{1,2,3\}\), we consider the action of \(sr^{a_{j}}\) and $r^{n/2}$ on vertices of $\Gamma$.
\begin{itemize}
    \item For \(i\in \mathbb{Z}_{n}\), we have
    \(r^{i} (sr^{a_{j}})=r^{i}(r^{-a_{j}}s)=r^{i-a_{j}}s=sr^{a_{j}-i}\).    
     So the edges produced by \(sr^{a_{j}}\) on rotation vertices are of the form \(\{r^{i},sr^{a_{j}-i}\}\). These edges form the perfect matching \(M_{a_{j}}\) between \(R_{n}\) and \(F_{n}\).

     \item For $y\in \mathbb{Z}_{n}$, we have
       \((sr^{y}) (sr^{a_{j}})=r^{a_{j}-y}\). Thus the edges produced by the generator \(sr^{a_{j}}\) on reflection vertices are of the form \(\{sr^{y},r^{a_{j}-y}\}\) for $y\in \mathbb{Z}_{n}$. If we set \(i=a_{j}-y\), then 
       $\{sr^{y},r^{a_{j}-y}\}=\{sr^{a_{j}-i},r^{i}\}\in M_{a_{j}}$.
       Thus, the edges generated by \(sr^{a_{j}}\) acting on reflection vertices are already defined in \(M_{a_{j}}\).
       
       %, which is \(\{r^{i},sr^{a_{j}-i}\}\) since $\Gamma$ is an undirected graph. This shows that the edges generated by \(sr^{a_{j}}\) acting on reflection vertices are the same set of edges as those generated by \(sr^{a_{j}}\) acting on rotation vertices, which we have already defined as \(M_{a_{j}}\).

     \item For rotation \(r^{n/2}\), since \((r^{n/2})^2=e\) we have for each \(i\),
     $r^i r^{n/2}=r^{\,i+n/2}\in R_{n} \text{ and } (sr^i) r^{n/2}=sr^{\,i+n/2}\in F_{n}$, so \(r^{n/2}\) induces \(N_{R_{n}}\) on \(R_{n}\) and \(N_{F_{n}}\) on \(F_{n}\).
\end{itemize}

The set of edges in \(\Gamma \) is the union of the matchings induced by each generator in \(S\). The reflection generators \(sr^{a_{i}}\), $1\leq i\leq 3$ induce $M_{a_{i}}$. The rotation generator \(r^{n/2}\) induces \(N_{R_{n}}\) and \(N_{F_{n}}\).
Thus,
$E(\Gamma)=N_{R_{n}}\cup N_{F_{n}}\cup M_{a_1}\cup M_{a_2}\cup M_{a_3}$.
 \end{proof}
%%%%%%%%%%%%%%%%%%%%%%%%%%%%%%%%%%%%%%
\section{Acknowledgements} 
The author would like to thank the EK\"{O}P-24-4-II-ELTE-996 University Excellence scholarship program for the financial support.
%%%%%%%%%%%%%%%%%%%%%%%%%%%%%%%%%%%%

%
% ---- Bibliography ----
%

\end{document}